\documentclass[reqno,11pt,centertags]{article}
\usepackage{amsmath,amsthm,amscd,amssymb,latexsym,upref}
\date{\today}

\usepackage[T2A,OT1]{fontenc}
\usepackage[ot2enc]{inputenc}
\usepackage[russian,english]{babel}

\usepackage[all]{xy}

\usepackage{epic,eepicemu}

\newcommand{\D}{{\mathbb{D}}}
\newcommand{\N}{{\mathbb{N}}}
\newcommand{\R}{{\mathbb{R}}}

\renewcommand{\C}{{\mathbb{C}}}

\newcommand{\T}{{\mathbb{T}}}

\newcommand{\bbH}{{\mathbb{H}}}

\newcommand{\cF}{{\mathcal{F}}}

\newcommand{\cI}{{\mathcal{I}}}

\newcommand{\cP}{{\mathcal{P}}}

\newcommand{\zu}{z_{u_0}}
\newcommand{\zub}{\overline z_{u_0}}
\newcommand{\zi}{z_{\infty}}
\newcommand{\tzi}{\tilde z_{\infty}}

\renewcommand{\Re}{\operatorname{Re}}

\renewcommand{\Cap}{\mathrm{Cap}}
\renewcommand{\d}{\mathrm{d}}
\newcommand{\dx}{\mathrm{d}x}

\allowdisplaybreaks \numberwithin{equation}{section}
\newtheorem{theorem}{Theorem}[section]

\newtheorem{lemma}[theorem]{Lemma}
\newtheorem{proposition}[theorem]{Proposition}

\newtheorem{corollary}[theorem]{Corollary}

\theoremstyle{definition}

\newtheorem*{remark}{Remark}

\title{Szeg\H o-Widom asymptotics of Chebyshev Polynomials on Circular Arcs}
\date{\today}
\author{Benjamin Eichinger}
\begin{document}
	\maketitle
	\begin{abstract}
Thiran and Detaille give an explicit formula for the asymptotics of the sup-norm of the Chebyshev polynomials on a circular arc. We give the so-called Szeg\H o-Widom asymptotics for this domain, i.e.,  explicit expressions for the asymptotics of the corresponding extremal polynomials. Moreover, we solve a similar problem with respect to the upper envelope of a family of polynomials uniformly bounded on this arc. That is, we give explicit formulas for the asymptotics of the error of approximation as well as of the extremal functions. Our computations show that in the proper normalization the limit of the upper envelope represents the diagonal of a reproducing kernel of a certain Hilbert space of analytic functions. Due to Garabedian the analytic capacity in an arbitrary domain is the diagonal of the corresponding Szeg\H o kernel. We don't know any result of this kind with respect to upper envelopes of polynomials. If this is a general fact or a specific property of the given domain, we rise as an open question.
	\end{abstract}
	\section{Introduction}
Getting explicit asymptotics is a fundamental problem in constructive approximation theory. The problem can be related to both, approximation error and the function of best approximation. Usually, the second problem is essentially harder. For instance, it took almost 100 years that Lubinsky \cite{Lub07} characterized the extremal function in the famous Bernstein problem on the best approximation in $[-1,1]$ of $|x|^\alpha$ by polynomials. A problem, which is closer related to the one that we are going to consider in this paper was recently solved by Christiansen, Simon and Zinchenko \cite{ChrSiZi15} (to appear in   Invent. math.). They showed that if $E\subset\R$ is regular, compact and satisfies the Parreau-Widom condition, then the Chebyshev polynomials obey $\|T_n\|_E\leq Q\Cap(E)^n$, where $\Cap(E)$ is the logarithmic capacity of $E$ and $\|\cdot\|_E$ denotes the sup-norm. In this case $\overline{\C}\setminus E$ may be infinitely connected. Under the restriction that $E$ is a finite union of intervals, they were also able to describe the asymptotics of the extremal functions $T_n$.

For further references on Chebyshev polynomials and its asymptotics see \cite{SodYud92,Szeg24,Tot09,Tot14,Wid69}.
We consider Chebyshev polynomials on circular sets $A_\alpha$, of the form 
\begin{align*}
A_{\alpha}=\{u\in\C:\ |u|=1,\ -\alpha\leq\arg u\leq \alpha\},\quad 0<\alpha< \pi.
\end{align*}
It is motivated by a paper of Thiran and Detaille \cite{ThirDet91}, who showed that the extremal value obeys
\begin{align}\label{eq:ThirDetasymptotics}
	\|T_n\|_{A_\alpha}\sim\cot(\alpha/4)\Cap(A_\alpha)^{n+1}.
\end{align}
Our approach is completely different and allows us to find:
\begin{itemize}
	\item[(i)] explicit asymptotics of the Chebyshev polynomials,
	\item[(ii)] explicit asymptotics of the upper envelope of the family $\cP_{n,\alpha}$ of polynomials of degree at most $n$ which are bounded by one in modulus on $A_\alpha$; cf. \eqref{def:Pnalpha}.
\end{itemize}
%
To be more precise, let $g_\Omega(z,z_0)$ denote the Green's function of the point $z_0$ and the domain $\Omega$. Writing  $i*g_\Omega(z,z_0)$ for the harmonic conjugate of $g_\Omega(z,z_0)$ we define the complex Green's function of the domain by
\begin{align*}
	b_\Omega(z,z_0)=e^{-(g_\Omega(z,z_0)+i*g_\Omega(z,z_0))};
\end{align*}
cf. \cite{Wid69}.
Instead of the Chebyshev polynomials, let us consider the normalized polynomials
$
	P_{n,\infty},
$
i.e., the polynomial in $\cP_{n,\alpha}$ that has maximal leading coefficient. Set $\Omega_\alpha=\overline{\C}\setminus A_\alpha$. Due to Montel's theorem, at least by passing to subsequences, the family $b_{\Omega_\alpha}(u,\infty)^nP_{n,\infty}(u)$ has a limit as $n\to\infty$. In Theorem \ref{thm:main} we present the limit function explicitly. Following the notion introduced in  \cite{ChrSiZi15},  we say that $\Omega_\alpha$ has Szeg\H o-Widom asymptotics. 

The leading coefficient reflects the behavior of the polynomial at $\infty$. It is therefore natural to consider this problem also for other points $u_0\in\Omega_{\alpha}$. By $P_{n,u_0}\in\cP_{n,\alpha}$ we denote those polynomials in $\cP_{n,\alpha}$ which have maximal value at the point $u_0$. Due to the symmetry of the domain, it suffices to consider the problem for $|u_0|<1$; see \eqref{eq:symmetryrelation}. 

Let $\lambda:\Omega_\alpha\to \Pi=\{\lambda\in\C:\  -\frac{\pi}{4}\leq \arg\lambda \leq \frac{\pi}{4}\}$  be defined by 
\begin{align}\label{def:lambda}
	\lambda(u)=\Big(\frac{ue^{i\alpha}-1}{u-e^{i\alpha}}\Big)^{1/4}.
\end{align}
\begin{theorem}\label{thm:szegoWidomDisc}
	Let $|u_0|<1$ and $\lambda_0=\lambda(u_0)$. There exists $\phi\in\R$ such that 
	\begin{align}\label{eq:formulaWidomSzegoU0}
	\lim\limits_{n\to\infty}b_{\Omega_\alpha}(u,\infty)^nP_{n,u_0}(u)=e^{i\phi}\frac{1}{2}\left(1+\frac{h(\lambda,\lambda_0)}{h(\lambda_0,\lambda_0)}\right)\frac{\lambda^2-|\lambda_0|^2}{\lambda^2+|\lambda_0|^2}\frac{\lambda^2+\lambda_0^2}{\lambda^2+\overline{\lambda_0}^2}
	\end{align} 
	uniformly on compact subsets of $\Omega_\alpha$, where
	\begin{align*}
		h(\lambda,\lambda_0)=\frac{\lambda^2}{(\lambda^2-|\lambda_0|^2)(\lambda^2+|\lambda_0|^2)}.
	\end{align*}
\end{theorem}	
Let $L_n(u)$ denote the upper envelope of $\cP_{n,\alpha}$, i.e., 
\begin{align*}
	L_n(u):=\sup\{|P_n(u)|:\ P_n\in\cP_{n,\alpha}\}.
\end{align*}
\begin{theorem}\label{thm:ReproducingKernel}
	Let $\lambda=\lambda(u)$ be defined as above, $\lambda_0=\lambda(u_0)$ and define
	the reproducing kernel $k_{\Omega_{\alpha}}(u,u_0)$ by
	\begin{align*}
	k_{\Omega_{\alpha}}(u,u_0)=k_{\bbH_+}(\lambda,\lambda_0):=\frac{2\lambda\overline{\lambda}_0}{(\lambda+\overline \lambda_0)^2}.
	\end{align*}
	Then 
	\begin{align}\label{eq:ExtremalValueIsReprod}
	L_n(u)\sim e^{ng_{\Omega_\alpha}(u,\infty)}k_{\Omega_{\alpha}}(u,u).
	\end{align}
\end{theorem}
Finally, we would like to mention that the proofs will show that these results are universal in the following sense. In fact, one could instead of polynomials consider rational functions with fixed collection of poles $C=\{c_1,\dots, c_g\}$ outside of $\D$. The solution for the same problem for the class $\cF_{n,\alpha}$ of rational functions with its only poles in $C$ of order at most $n$ is denoted by $F_{n,u_0}$. Let $B(u)=\prod b_{\Omega_\alpha}(u,c_k)$. The limit 
\begin{align*}
	\lim\limits_{n\to\infty}B(u)^nF_{n,u_0}(u)=e^{i\phi}\frac{1}{2}\left(1+\frac{h(\lambda,\lambda_0)}{h(\lambda_0,\lambda_0)}\right)\frac{\lambda^2-|\lambda_0|^2}{\lambda^2+|\lambda_0|^2}\frac{\lambda^2+\lambda_0^2}{\lambda^2+\overline{\lambda_0}^2},
\end{align*}
for every choice of $C$. In particular, the upper envelope of this family, denoted by $M_n$, satisfies
\begin{align*}
	M_n(u)\sim e^{n\sum g_{\Omega_\alpha}(u,c_k)}k_{\Omega_{\alpha}}(u,u).
\end{align*}
	\section{Szeg\H o -Widom asymptotics}
In order to have uniqueness, we fix the normalization for the extremal polynomials and the complex Green's function
\begin{align*}
\lim_{u\to\infty}ub_{\Omega_\alpha}(u,\infty)>0\quad\text{ and }\quad b_{\Omega_\alpha}(u_0,\infty)^nP_{n,u_0}(u_0)>0.
\end{align*}
Let $\cP_n$ be the set of all polynomials of degree at most $n$ and 
\begin{align}\label{def:Pnalpha}
\cP_{n,\alpha}=\{P\in\cP_n: \|P\|_{A_\alpha}\leq 1\}.
\end{align}
Since $A_\alpha\subset\T$, the map 
\begin{align}\label{def:starmap}
P_n(u)\mapsto P_n^*(u):=u^n\overline{P_n\left (1/\overline{u}\right )},
\end{align}
is an involution on $\cP_{n,\alpha}$. This shows that for $L_n(\infty):=1/\|T_n\|_{A_\alpha}$ we have
\begin{align}\label{eq:leastDevEvalAtOrigin}
L_n(\infty)=L_n(0)
\end{align}
and there exists $\phi\in\R$ such that $P_{n,\infty}=e^{i\phi}P^*_{n,0}$. We will give a solution of \eqref{eq:leastDevEvalAtOrigin} by reducing it to a problem which was already considered by Yuditskii \cite{Yud99}. Let $A_0=\R\setminus (-1,1)$ and $\Omega_0=(\C\setminus\R)\cup (-1,1)$. The map
\begin{align*}
u(z)=\frac{z-z_0}{z- \overline z_0},\qquad z_0=i\tan(\alpha/2),
\end{align*}
maps $\Omega_0$ conformally onto $\Omega_\alpha$. By $z(u)$ we denote its inverse map. Henceforth, if we use $z$ and $u$ simultaneously we have in mind $z(u)$ and $u(z)$, respectively.  Defining $z_\infty=z(\infty)$ it is obvious that  $\zi=\overline z_0$. Note that $z(e^{i\alpha})=-1$ and $z(e^{-i\alpha})=1$. To the polynomial $E_n(z)=(z-z_\infty)^n$, we associate the weighted norm 
\begin{align}\label{def:piNorm}
	\|Q_n\|_{\Pi(E_n)}:=\sup_{x\in A_0}\left|\frac{Q_n(x)}{E_n(x)}\right|\quad\text{for }Q_n\in\cP_n.
\end{align}
\begin{lemma}
Let $Q_{n,z_0}\in\cP_n$ be the solution of the extremal problem
\begin{align*}
	|Q_{n,z_0}(z_0)|=\sup\{|Q_n(z_0)|:\ Q_n\in\cP_n, \|Q_n\|_{\Pi(E_n)}\leq 1\}.
\end{align*}
Then there exists $\phi\in\R$ such that
\begin{align}\label{eq:QPRelation}
	P_{n,0}(u)=e^{i\phi}\frac{Q_{n,z_0}(z)}{E_n(z)}.
\end{align}
\end{lemma}
\begin{proof}
	We have
	\begin{align*}
		C\frac{z-z_l}{z-z_\infty}=u(z)-u(z_l),\quad C=\frac{z_0-\zi}{z_l-\zi}.
	\end{align*}
	Hence, the map
	\begin{align*}
		P_n(u)\mapsto Q_n(z):=E_n(z)P_n(u(z)),
	\end{align*}
	maps $\cP_n$ bijectively onto itself. Moreover,
	\begin{align*}
		\|P_n\|_{A_\alpha}=\sup_{x\in A_0}\left|\frac{Q_n(x)}{E_n(x)}\right|=\|Q_n\|_{\Pi(E_n)}.
	\end{align*}
	Therefore,
	\begin{align*}
		|P_{n,0}(0)|&=\sup\{|P_n(0)|:\ P_n\in\cP_{n,\alpha}\}\\
		&=\sup\left \{\left|\frac{Q_n(z_0)}{E_n(z_0)}\right|:\ Q_n\in\cP_n,\|Q_n\|_{\Pi(E_n)}\leq 1\right\}
		=\left|\frac{Q_{n,z_0}(z_0)}{E_n(z_0)}\right|
	\end{align*}
	and \eqref{eq:QPRelation} holds.
\end{proof}
In \cite{Yud99} an explicit solution for this kind of problem is given. First, let us mention that in \cite{ThirDet91} it is shown that for fixed $\alpha$ there may be $N\in\N$ such that for $n<N$ the extremal polynomial is just $z^n$. This corresponds to a special case in \cite{Yud99}. Since we are only interested in asymptotics we assume that $n>N$. We recall the theorem in a way that is convenient for our purpose. Let $\omega(z,I;\Omega)$ denote the harmonic measure of the domain $\Omega$.
\begin{theorem}[\cite{Yud99}]\label{thm:maincCP}
	Let $E_n(z)$ be a polynomial with zeros $Z=\{\overline z_1,\dots,\overline z_n\}\subset\C_-$ and $z_0=ir\in i\R_{>0}$. Then there exists a unique $0<x_n<1$ such that $I_n=[-x_n,x_n]$ satisfies
	\begin{align}\label{eq:harmonicMeasureSum}
		\sum_{\overline z_l\in Z\cup\{\overline z_0\}}\omega(\overline{z}_l,I_n;\Omega_0\setminus I_n)=1.
	\end{align}
	Let $\Omega_n=\Omega_0\setminus I_n$ and set
	\begin{align*}
		\cI(z)=\prod_{\overline z_l\in Z\cup \{\overline z_0\}}b_{\Omega_n}(z,\overline z_l)
	\end{align*}
	and
	\begin{align*}
		s_n(z)=\sqrt{\frac{z_0^2-1}{z_0^2-x_n^2}\frac{z^2-x_n^2}{z^2-1}},\quad s_n(z_0)=1.
	\end{align*}
	The extremal polynomial $Q_{n,z_0}$ is up to the unimodular factor $e^{i\phi}$ uniquely given by
	\begin{align*}
	Q_{n,z_0}(z)=e^{i\phi}E_n(z)\left(\frac{1+s_n(z)}{2s_n(z)}\frac{1}{\cI(z)}+\frac{1-s_n(z)}{2s_n(z)}\frac{z-z_0}{z-\overline z_0}\frac{\overline{E_n(\overline{z})}}{E_n(z)}\cI(z)\right).
	\end{align*}
	In particular,
	\begin{align*}
	L_n(z_0)=|E_n(z_0)|\exp\bigg(\sum_{z\in Z\cup \{\overline z_0\}}g_{\Omega_n}(\overline{z}_l,z_0)\bigg).
	\end{align*}
\end{theorem}
Hence, by Theorem \ref{thm:maincCP} the solution of the extremal problem is given by
\begin{align}\label{def:Qn}
Q_{n,z_0}(z)=E_n(z)\left(\frac{1+s_n(z)}{2s_n(z)}b_{n}(z)^{-(n+1)}+\frac{1-s_n(z)}{2s_n(z)}\frac{(z- z_0)^{n+1}}{(z-\overline z_0)^{n+1}}b_{n}(z)^{n+1}\right),
\end{align}
where $b_n(z)=b_{\Omega_n}(z,\zi)$. We also abbreviate 
$g_n(z)=g_{\Omega_n}(z,\zi)$, $g(z)=g_{\Omega_0}(z,\zi)$, $b(z)=b_{\Omega_0}(z,\zi)$ and $\omega_n(E)=\omega(\zi,E;\Omega_n)$.
Note that \eqref{eq:harmonicMeasureSum} reads
\begin{align}\label{eq:harmonicMeasureSum2}
	\omega_n(I_n)=\frac{1}{n+1}
\end{align}
in this case. 
Our goal is to find the limit of $b_{\Omega_\alpha}(u,\infty)^nP_{n,0}(u)$. Due to the conformal invariance of the Green's function, this is equivalent to finding the asymptotics of 
\begin{align*}
f_n(z)=\frac{b(z )^nQ_{n,z_0}(z)}{E_n(z)}.
\end{align*}
By the maximum principle and Montel's theorem, there exist subsequences $n_j$ such that $f_{n_j}$ converges to an analytic function $f$ uniformly on compact subsets of $\Omega_0$. We will show that all subsequences have the same limit.
\begin{lemma}\label{lem:ZeroPart}
	Let $I_n=[-x_n,x_n]$. Then $x_n\to 0$ as $n\to \infty$ and 
	\begin{align*}
	\lim_ns_n(z)^2=\frac{z_0^2-1}{z_0^2}\frac{z^2}{z^2-1},
	\end{align*}
	where the limit is uniformly on compact subsets of $\Omega_0$. Moreover, 
	\begin{align*}
	\lim\limits_{n}\frac{1-s_n(z)}{2s_n(z)}\frac{(z- z_0)^{n+1}}{(z-\overline z_0)^{n+1}}b(z)^nb_{n}(z)^{n+1}=0,
	\end{align*}
	uniformly on compact subsets of $\Omega_0\setminus\{0\}.$
\end{lemma}
\begin{proof}
	By the maximum principle (see \cite[12 Ch. IV, Sec 2]{Nev70}) $\omega_n(I_n)$ is an increasing function of $x_n$. Since $\omega_n(I_n)\to 0$ as $x_n\to 0$, we obtain the first statement and the second statement is clear. 
	Finally, we notice that on a compact subset of $\Omega_0\setminus\{0\}$ 
	\begin{align*}
	\lim\limits_{n}\frac{1-s_n(z)}{2s_n(z)}=\frac{1-s(z)}{2s(z)},
	\end{align*}
	which is analytic there. Since 
	\begin{align*}
	\left|\frac{(z- z_0)}{(z-\overline z_0)}b_{n}(z)b(z)\right|<1\quad \text{on }\Omega_n
	\end{align*}
	we obtain the last statement.
\end{proof}
\begin{theorem}\label{thm:main}
	The domain $\Omega_\alpha$ has Szeg\H o-Widom asymptotics. That is, there exists $\phi\in\R$ such that 
	uniformly on compact subsets of $\Omega_\alpha$ we have
	\begin{align}\label{eq:WidomSzegoQ}
		\lim\limits_{n\to\infty}\overline{b_{\Omega_\alpha}(u^*,\infty)^nP_{n,\infty}(u^*)}=e^{i\phi}\frac{1+s(z)}{2s(z)}\frac{b_{\Omega_0}(z,0)}{b_{\Omega_0}(z,\overline{z}_0   )},
	\end{align}
	where 
	\begin{align*}
		s(z)=\sqrt{\frac{z_0^2-1}{z_0^2}\frac{z^2}{z^2-1}},\quad s(z_0)=1.
	\end{align*}
\end{theorem}
\begin{proof}
	Solving the Dirichlet problem for the harmonic function 
	\[
	h(z_1)=g_{\Omega_0}(z_1,z)-g_{\Omega_n}(z_1,z)
	\] 
	in  $\Omega_n$ shows 
	\[ 
	h(z_1)=\int_{I_n}g_{\Omega_0}(z,x)\omega(z_1,\dx;\Omega_n).
	 \]
	The symmetry of the Green's function with respect to the variables $z$ and $z_1$ leads to
	\begin{align}\label{eq:FormulaGreen}
	g_{\Omega_0}(z,z_1)-g_{\Omega_n}(z,z_1)=\int_{I_n}g_{\Omega_0}(z,x)\omega(z_1,\dx;\Omega_n).
	\end{align}
	Therefore,
	\begin{align*}
		\log\left|\frac{b(z)^n}{b_n(z)^n}\right|=n(g_n(z)-g(z))=-\int_{I_n}g_{\Omega_0}(z,x)n\omega_n(\dx).
	\end{align*}
	By \eqref{eq:harmonicMeasureSum2}, $\chi_{I_n}n\omega_n(\dx)$ converges to the delta distribution and therefore
	\begin{align*}
		\lim\limits_{n\to\infty}\log\left|\frac{b(z)^n}{b_n(z)^n}\right|= -g_{\Omega_0}(z,0).
	\end{align*}
	In the same way we see that
	$
		\lim_{n\to\infty}\log|b_{\Omega_n}(z,\overline z_0)|=-g_{\Omega_0}(z,\overline z_0).
	$
	Therefore, in combination with Lemma \ref{lem:ZeroPart} we obtain for the limit function $f$ that 
	\begin{align*}
		|f(z)|=\left|\frac{1+s(z)}{2s(z)}\frac{b_{\Omega_0}(z,0)}{b_{\Omega_0}(z,\overline{z}_0)}\right|,
	\end{align*}
	and hence
	\begin{align*}
			f(z)=\frac{1+s(z)}{2s(z)}\frac{b_{\Omega_0}(z,0)}{b_{\Omega_0}(z,\overline{z}_0)}.
	\end{align*}
	This shows
	\begin{align}\label{WidomSzegoQinZero}
		\lim_{n\to\infty}\frac{b(z)^nQ_{n,z_0}(z)}{E_n(z)}=e^{i\phi}\frac{1+s(z)}{2s(z)}\frac{b_{\Omega_0}(z,0)}{b_{\Omega_0}(z,\overline{z}_0)}.
	\end{align}
	Due to the symmetry of the domain with respect ot the real line we have   $\overline{b_{\Omega_\alpha}(\overline{u},\infty)}=b_{\Omega_\alpha}(u,\infty)$. This and the conformal invariance of the Green's function leads to
	\begin{align*}
		b_{\Omega_\alpha}(u,\infty)^nP_{n,\infty}(u)&=b_{\Omega_\alpha}(u,\infty)^ne^{i\phi}u^n\overline{P_{n,0}(1/\overline{u})}\\
		&=	e^{i\phi}b_{\Omega_\alpha}(u,0)^n\overline{P_{n,0}(u^*)}\\
		&=e^{i\phi}\overline{b_{\Omega_\alpha}(u^*,\infty)^nP_{n,0}(u^*)}.
	\end{align*}
	This concludes the proof.
\end{proof}
As a corollary of Theorem \ref{thm:main}, we obtain \eqref{eq:ThirDetasymptotics}. Recall that 
\begin{align*}
f_n\sim g_n\quad\iff\quad \lim\limits_{n}\frac{f_n}{g_n}=1
\end{align*}
and
\begin{align*}
	\Cap(A_\alpha):=\lim_{u\to\infty}|ub_{\Omega_\alpha}(u,\infty)|.
\end{align*}
\begin{corollary}\label{cor:assymptoticsLeastdevi}
	Let $T_n$ denote the Chebyshev polynomials of $A_\alpha$. Then 
	\begin{align*}
	\|T_n\|_{A_\alpha}\sim\cot(\alpha/4)\Cap(A_\alpha)^{n+1}.
	\end{align*}
\end{corollary}
\begin{proof}
	Due to \eqref{eq:leastDevEvalAtOrigin}, 
	\begin{align*}
	L_n(0)=\frac{1}{\|T_n\|_{A_\alpha}}.
	\end{align*}
	Let $w:\Omega_0\to\C_+$ be defined by $w(z)=\sqrt{\frac{z-1}{z+1}}$ and $w(z_0)=w_0=e^{i(\pi-\alpha)/2}$. Since $w(0)=i$, we obtain
	\begin{align*}
		|b_{\Omega_0}(z_0,0)|=|b_{\C_+}(w_0,i)|=\left|\frac{e^{i(\pi-\alpha)/2}-i}{e^{i(\pi-\alpha)/2}+i}\right|=\tan(\alpha/4).
	\end{align*}
	 Moreover, 
	\begin{align*}
	\Cap(A_\alpha)&=|\lim_{u\to\infty}ub_{\Omega_\alpha}(u,\infty)|=|\lim_{u\to\infty}b_{\Omega_\alpha}(u,0)|=|b_{\Omega_\alpha}(\infty,0)|.
	\end{align*}
	Thus, Theorem \ref{thm:main} shows that
	\begin{align*}
	\lim\limits_{n\to\infty}\left|\frac{\Cap(A_\alpha)^{n}}{\|T_n\|_{A_\alpha}}\right|=\frac{\tan(\alpha/4)}{\Cap(A_\alpha)},
	\end{align*}
	which concludes the proof.
\end{proof}
The next natural question is to solve this problem not only for $u_0=0$, but for an arbitrary point $u_0\in\Omega_\alpha$. As before, due to the symmetry of the domain, we can reduce it to $u_0\in \D$. Namely, if $|u_0|>1$, we have
\begin{align}\label{eq:symmetryrelation}
P_{n,u_0}=P^*_{n,u_0^{*}},\quad b_{\Omega_\alpha}(u,\infty)^nP_{n,u_0}(u)=\overline{b_{\Omega_\alpha}(u^*,\infty)^nP_{n,u_0^*}(u^*)}.
\end{align}
\begin{lemma}
	Let $|u_0|<1$ and $z_{u_0}=z(u_0)$. Let $K_0$ be the unique circle that passes through $\zu$ and $\zub$ such that $\Omega_0$ is symmetric with respect to reflection by $K_0$. Moreover, let ${x_0}=K_0\cap (-1,1)$ and
	\begin{align*}
		s(z,\zu)=\sqrt{\frac{\zu^2-1}{(\zu-x_0)^2}\frac{(z-x_0)^2}{z^2-1}}, \quad s(\zu,\zu)=1.
	\end{align*}
	Then there exists $\phi\in\R$ such that uniformly on compact subsets of $\Omega_\alpha$
	\begin{align*}
		\lim\limits_{n\to\infty}b_{\Omega_\alpha}(u,\infty)^nP_{n,u_0}(u)=e^{i\phi}\frac{1+s(z,\zu)}{2s(z,\zu)}\frac{b_{\Omega_0}(z,x_0)}{b_{\Omega_0}(z,\zub)}.
	\end{align*} 
\end{lemma}	
\begin{proof}
	Let $u_0\in\D$. By a M\"obius transformation $\psi$, (aka Blaschke factor of the disc) we map $u_0\mapsto 0$ such that $A_\alpha$ is mapped onto $A_{\alpha'}$ for some $\alpha'$, i.e.,  $A_{\alpha'}$ is still symmetric with respect to the real axis and $1\in A_{\alpha'}$. Then we compose this map with $z$ (related to $\alpha'$) of the previous section in order to obtain a conformal map $\tilde z:\Omega_\alpha\to\Omega_0$ such that
	\begin{align*}
	\tilde z(e^{i\alpha})=-1,\quad  \tilde z(e^{-i\alpha})=1,\quad \tilde{z}(u_0)=i\tan(\alpha'/2).
	\end{align*} 
	Hence, we can apply exactly the same procedure in proving the asymptotics
	\begin{align*}
	\lim\limits_{j}\frac{b(\tilde z,\tzi;\Omega_0)^{n_j}\tilde Q_{n_j,\tilde z_0}(\tilde z)}{\tilde E_{n_j}(\tilde z)}=\frac{1+ \tilde s(\tilde z)}{2\tilde s(\tilde z)}\frac{b(\tilde z,0)}{b(\tilde z,\overline{\tilde z(u_0)})},
	\end{align*}
	where 
	\begin{align*}
		\tilde s(\tilde z)^2=\frac{\tilde z(u_0)^2-1}{\tilde z(u_0)^2}\frac{\tilde z^2}{\tilde z^2-1}.
	\end{align*}
	The map $\phi:\Omega_0\to\Omega_0$ with $\phi(\tilde z)=z$ is a fractional linear transformation (FLT) with $\phi(\R)=\R$. 
	\begin{align*}
	\begin{xy}
	\xymatrix{
		\overline{\C}\setminus A_\alpha\ar^z[r]\ar^\psi[d]&  \Omega_0\\
		\overline{\C}\setminus A_\alpha\ar^{\tilde z}[r]& \Omega_0\ar^\phi[u]
	}
	\end{xy}
	\end{align*}
	 Due to properties of conformal maps in particular of FLTs we obtain $\phi(i\R)=K_0$, $\phi(0)=x_0$, $\phi\left(\tilde z(u_0)\right)=\zu$ and $\phi\left(\overline{\tilde z(u_0)}\right)=\zub$, which concludes the proof.
\end{proof}
	\begin{proof}[Proof of Theorem \ref{thm:szegoWidomDisc}]
		The function $\lambda$ given by \eqref{def:lambda} is a composition of the maps $z:\Omega_\alpha\to \Omega_0$, $w:\Omega_0\to \C_+$, defined by
		$
		w(z)=\sqrt{\frac{z-1}{z+1}}
		$
		and $\tilde \lambda (w):\C_+\to\Pi$ defined by $\lambda(w)=\sqrt{-i w}$. Let $w_0=w(z_{u_0})$. Using the reflection principle and that FLTs map circles onto circles, we obtain that $w(x_0)=i|w_0|$ and $w(\zub)=-\overline{w}_0$. Hence,
		\begin{align*}
				\lim\limits_{n\to\infty}b_{\Omega_\alpha}(u,\infty)^nP_{n,u_0}(u)=e^{i\phi}\frac{1}{2}\left(1+\frac{v(w,w_0)}{v(w_0,w_0)}\right)\frac{w-i|w_0|}{w+i|w_0|}\frac{w+w_0}{w+\overline{w}_0},
		\end{align*}
		where 
		$
		v(w,w_0)=\frac{w}{(w+i|w_0|)(w-i|w_0|)}.
		$
		By definition $w=i\lambda^2$, which proves \eqref{eq:formulaWidomSzegoU0}.
	\end{proof}
 We define 
 \begin{align*}
 	L(u):=\lim_n e^{-ng_{\Omega_\alpha}(u,\infty)}L_n(u).
 \end{align*}
 Note that \eqref{eq:symmetryrelation} in particular implies that $L(u_0^*)=L(u_0)$. 
 
 Let us point out that the fact that we don't give a formula for $P_{n,u_0}$ for $|u_0|=1$ is just a consequence of our technique. Indeed, this question leads to a real Chebyshev problem, which was already introduced and solved by Chebyshev \cite{Cheb59}. Later this problem was widely discussed; see e.g. \cite{Akh53,Markov}. We only use that for each $n$ there exists a maximizer $P_{n,u_0}$(either by referring to the real Chebyshev problem or by compactness of $\cP_{n,\alpha}$). Due to Montel's theorem, $b_{\Omega_\alpha}(u,\infty)^nP_{n,u_0}(u)$ has a convergent subsequence, i.e., there exists $f(u)$ such that
 $
 f(u)=\lim_{j}b_{\Omega_\alpha}(u,\infty)^{n_j}P_{n_j,u_0}(u).
$
 Set
$
 L(u_0)=f(u_0).
$
 We will see that with this definition $L(u)$ is continuous and therefore this value is independent of the particular choice of the subsequence. 
 \begin{lemma}
 	$L_n(u)$ and $L(u)$ are continuous on $\C\setminus A_\alpha$ and $\Omega_\alpha$, respectively.
 \end{lemma}
 \begin{proof}
	Let $P\in P_{n,\alpha}$. Since $|b_{\Omega_\alpha}(u,\infty)^nP(u)|\leq 1$, $P$ is locally bounded and therefore $\cP_{n,\alpha}$ is equicontinuous. 
 	Hence, for every $u_0\in\C\setminus A_\alpha$ there exists $\delta>0$ such that $|u-u_0|<\delta$ implies
 	\begin{align*}
 	L_n(u_0)\geq |P_{n,u}(u_0)|>|P_{n,u}(u)|-\epsilon=L_n(u)-\epsilon.
 	\end{align*}
 	In the same way we see that $L_n(u)>L_n(u_0)-\epsilon$ and therefore $|L_n(u)-L_n(u_0)|< \epsilon$. 
 	The same proof applies for $L(u)$.
 \end{proof}

	\section{Log Subharmonicity and Reproducing Kernels}
In this chapter we will prove some properties of the extremal values $L_n(u)$ and $L(u)$ as functions on $\Omega_\alpha$. We recall the definition of log subharmonicity. 

Let  $\Omega\subset\C$ and $f:\Omega\to\R$ be an upper semicontinuous function. It is called subharmonic if for every $z_0\in\Omega$ there exists $R$ such that  $\{z:\ |z-z_0|\leq R\}\subset\Omega$ and for all $0<r\leq R$ we have
\begin{align*}
f(z_0)\leq\frac{1}{2\pi}\int_0^{2\pi}f(z_0+re^{it})\d t;
\end{align*}
cf. \cite{Duren70}. A function is called $\log$ subharmonic if $\log f$ is subharmonic.
\begin{remark}
	If $f$ is  twice continuously differentiable, then $f$ is subharmonic if and only if $\Delta f\geq 0$ in $\Omega$. 
\end{remark}
\begin{proposition}\label{prop: subharmonic}
	$L_n(u)$ and $L(u)$ are $\log$ subharmonic on $\C\setminus A_\alpha$ and $\Omega_\alpha$, respectively.
\end{proposition}
\begin{proof}
The modulus of an analytic function is $\log$ subharmonic. Since $L_n$ is continuous it can be easily seen that it is $\log$ subharmonic as the upper envelope of polynomials; cf. \cite[Lecture 7]{Lev96}. Clearly, this also holds for $|b_{\Omega_\alpha}(u,\infty)|^nL_n(u)$. Note that
\begin{align*}
\log|b_{\Omega_\alpha}(u,\infty)^nL_n(u)|=n(g_n(z)-g(z))+g_{\Omega_n}(z,\overline{z}_{u_0})
\end{align*}
By the maximum principle $g_{\Omega_n}(z,z_1)$ is increasing in $n$ and therefore this holds for $n(g_n(z)-g(z))+g_{\Omega_n}(z,\zub)$.  Thus, we can interchange limit and integration and obtain that $L(u)$ is $\log$ subharmonic.  
\end{proof}
\begin{proof}[Proof of Theorem \ref{thm:ReproducingKernel}]
	Evaluating \eqref{eq:formulaWidomSzegoU0} at $u_0$, we obtain
	\begin{align*}
		L(u_0)&=\left|\frac{\lambda_0^2-|\lambda_0|^2}{\lambda_0^2+|\lambda_0|^2}\right|\frac{2|\lambda_0|^2}{|\lambda_0^2-\overline{\lambda_0}^2|}\\
		&=\frac{|(\lambda_0^2-|\lambda_0|^2)(\overline{\lambda_0}^2+|\lambda_0|^2)|}{|\lambda_0^2+|\lambda_0|^2|^2}\frac{2|\lambda_0|^2}{|\lambda_0^2-\overline{\lambda_0}^2|}\\
		&=\frac{|\lambda_0|^2|\lambda_0^2-\overline{\lambda_0}^2|}{|\lambda_0|^2|\lambda_0+\overline{\lambda_0}|^2}\frac{2|\lambda_0|^2}{|\lambda_0^2-\overline{\lambda_0}^2|}\\
		&=\frac{2|\lambda_0|^2}{|\lambda_0+\overline{\lambda_0}|^2}=k_{\bbH_+}(\lambda_0,\lambda_0).
	\end{align*}
\end{proof}
	
\begin{remark}
	\begin{itemize}
		\item[(i)] Let $\partial,\overline{\partial}$ denote the Wirtinger derivatives. Since $4\partial\overline{\partial}=\Delta $,  we have for twice  continuously differentiable functions that
		\begin{align*}
		f \text{ is }\log\text{subharmonic} \iff \begin{bmatrix}
		f(z)& \overline{\partial} f(z)\\
		\partial f(z)& \partial\overline{\partial}  f(z)
		\end{bmatrix}
		\geq 0.
		\end{align*}
		Note that this matrix-inequality, which appears naturally for $L$ as limit of an upper envelope of polynomials is just a small part of an matrix-inequality, which holds for reproducing kernels of analytic functions. Namely, writing 
		\begin{align*}
		k(z,z_0)=\sum_k\phi_k(z)\overline{\phi_k(z_0)},
		\end{align*}
		for an orthonormal basis $\{\phi_k\}$, we see that the matrix $\{\partial^i\overline{\partial}^jk(z_0,z_0)\}_{i,j=1}^n$ is the Gram matrix of the vectors
		$\{\phi_k(z_0)\}, \{\partial\phi_k(z_0)\},\dots, \{\partial^n\phi_k(z_0)\}$ with respect to the standard $\ell^2$ scalar product and therefore 
		$$
		\{\partial^i\overline{\partial}^jk(z_0,z_0)\}_{i,j=1}^n\geq 0\quad \text{for all } n\in\N. 
		$$ 
		\item[(ii)] The kernel
		$
			\frac{1}{(\lambda+\overline \lambda_0)^2}
		$
		is up to normalization that Bergman kernel of $\bbH_+=\{z\in\C:\ \Re z>0\}$. Hence, $k_{\bbH_+}(\lambda,\lambda_0)$ is the reproducing kernel for the weighted Bergman space $\lambda A^2(\bbH_+)$. 
		\item[(iii)] Since $\Pi$ is a subset of $\bbH_+$ we assume that 	$k_{\Omega_{\alpha}}(u,u_0)$ has an extension to some larger domain (probably a Riemann surface). This is quite understandable for the following reason. In general, a reproducing kernel $k(z_0,z_0)$ diverges if $z_0$ converges to the boundary of the domain. But due to our setting, it is clear that $L(u)\to 1$ as $u\to A_\alpha$. Hence, $A_\alpha$ might not be real boundary of the domain on which $k_{\Omega_{\alpha}}(u,u_0)$ is defined.
\end{itemize}		
\end{remark}
	 \bibliographystyle{amsplain}
\bibliography{lit}
\end{document}